\documentclass[12pt]{amsart}

\usepackage{amssymb,amsmath,amsthm,hyperref}
\hypersetup{
pdftitle={Research Paper},
pdfauthor={Abbas Nasrollah Nejad},
pdfsubject={Linear type property of Jacobian ideal of a hypersurface},
pdfcreator={Abbas Nasrollah Nejad},
colorlinks,
linkcolor=blue,
citecolor=magenta,
anchorcolor=red,
bookmarksopen,
urlcolor=red,
filecolor=red,
pdfpagetransition={Wipe}
}

\theoremstyle{plain}
\newtheorem{Theorem}{Theorem}[section]
\newtheorem{Corollary}[Theorem]{Corollary}
\newtheorem{Lemma}[Theorem]{Lemma}
\newtheorem{Proposition}[Theorem]{Proposition}
\newtheorem{Conjecure}[Theorem]{Conjecture}
\newtheorem{Question}[Theorem]{Question}
\theoremstyle{definition}
\newtheorem{Definition}[Theorem]{Definition}
\newtheorem{Example}[Theorem]{Example}
\theoremstyle{Remark}
\newtheorem{Remark}{Remark}

\newcommand{\lar}{\longrightarrow}

\def\fm{{\mathfrak m}}

\def\fp{{\mathfrak p}}
\def\fq{{\mathfrak q}}

\def\hht{{\rm ht}\,}

\def\hom{\mbox{\rm Hom}}

\def\proj#1{{\rm Proj}\, (#1)}

\def\pp{{\mathbb P}}
\def\AA{{\mathbb A}}

\begin{document}

\title[Hypersurfaces with linear type singular loci]{Hypersurfaces with linear type singular loci }
\author[A. Nasrollah Nejad, A.B. Farrahy  ]{Amir Behzad Farrahy, \ abbas Nasrollah nejad   }
\address{department of mathematics,  institute for advanced studies in basic sciences (IASBS) p.o.box 45195-1159  zanjan, iran.
}
\email{farrahy@iasbsa.ac.ir, abbasnn@iasbs.ac.ir}

\subjclass[2010]{primary 14J70, 14Q10, 14H20; secondary  13A30, 14M10, 13H10}
\keywords{ Hypersurface, Blowup algebra, Jacobian Ideal, Isolated singularity}
\begin{abstract}
 In this paper, necessary and sufficient criteria for the Jacobian ideal of a reduced hypersurface with isolated singularity to be of linear type, are presented.  We prove that the gradient  ideal of a reduced projective plane curve with simple singularities ($\mathrm{ADE}$) is of linear type. We show that any reduced projective quartic curve is of gradient linear type.
 \end{abstract}
\maketitle
\section*{Introduction}
Let $R=k[x_1,\ldots,x_n]$ be a polynomial ring over an algebraically closed field $k$ of characteristic zero and  $f\in R$ a reduced polynomial defining a hypersurface $X=V(f)$. The singular subscheme of $X$ is defined by the {\it Jacobian ideal} $I(f)=(f,J(f))$, where $J(f)=(\partial f/\partial x_1,\ldots, \partial f/\partial x_n)$ is the \textit{gradient ideal} of $f$. This ideal has been mainly treated in its local or homogeneous version, largely profiting from the prevailing structure and theoretical simplification ({\rm cf.}\cite{CRA,ASimis,ss}).

An ideal $I$ in a commutative ring $A$ is called of \textit{linear type} if the symmetric algebra  $\mathrm{Sym}_I(A)$ of $I$ is isomorphic to the corresponding Rees algebra $\mathcal{R}_I(A)$. We say that a hypersurface $X=V(f)$ is of {\it Jacobian linear type } if the Jacobian ideal $I(f)$ is of linear type. If $X$ is of Jacobian linear type, then the blowup of the affine space $\mathbb{A}_k^n$ along its singular subscheme  $\proj{\mathcal{R}_{R}(I(f))}$ is the naive blowup $\proj{\mathrm{Sym}_{R}(I(f))}$. The basic notion which motivated this work is the {\it Aluffi algebra}.  This algebra is an algebraic definition of a characteristic cycle of a hypersurface in intersection theory.  Let $J\subset I$ be a pair of ideals in the ring $A$. The Aluffi algebra is defined by 
\[\mathcal{A}_{A/J}(I/J)={\rm Sym}_{A/I}(I/J)\otimes_{{\rm Sym}_A(I)}\mathcal{R}_A(I). \]
P. Aluffi used this algebra  for computing characteristic cycle of a hypersurface parallel to conormal cycle in intersection theory. In order to have expected behavior of the Aluffi algebra  over a regular ring, we had better to assume that the ideal $J$  is principal~\cite[Theorem 1.8 and Proposition 2.11]{AA}. If the ideal $I$ is of linear type, then the Aluffi algebra is isomorphic to the symmetric algebra. Hence  
\[\proj{\mathcal{A}_{A/J}(I/J)}=\proj{{\rm Sym}_{A/I}(I/J)}.\]
If the reduced hypersurface $X=V(f)$ is of Jacobian linear type, then  the characteristic cycle [{\rm Ch}(X)] of $X$ is equal $(-1)^{\dim X}\left[\proj{{\rm Sym}_{R/(f)}(I(f)/(f)}\right]$\cite[Theorem 3.2]{aluffi}.  In view of the results of Aluffi and of the application to characteristic cycle, it is natural to pose the problem of{ \it characterizing Jacobian linear type hypersurface}.  The results and examples in \cite{Abbas} hint that the answer to this problem may not be  simple in general. 

In this paper, we focus on the hypersurfaces with only isolated singularities. 
Let $X=V(f)\subseteq \mathbb{A}_k^n$ be a reduced affine hypersurface with only isolated singularities defining by the reduced polynomial $f\in k[x_1,\ldots,x_n]$. There are two algebras related to the singularities of $X$. Namely, the \textit{Milnor algebra} $M(f)=R/J(f)$ and \textit{Tjurina algebra} $T(f)=R/I(f)$. For each singular point $\fm$,  the local Milnor algebra $M(f)_{\fm}$ is Artinian Gorenstein ring. Note that the Tjurina algebra is the coordinate ring of singular subscheme of $X$. The numbers $\dim_k(M(f)_{\fm})$ and $\dim_k(T(f)_{\fm})$ are called  Milnor and Tjurina number of the hypersurface $X$ at singular point $\fm$. The hypersurface $X$ is called locally Eulerian if the Milnor number is equal the Tjurina number at each singular point.  In section 1, we prove that $X$ is of  Jacobian linear type if and only if $X$ is locally Eulerian hypersurface which is equivalent to say that  the Tjurina algebra is Artinian Gorenstein ring (Theorem \ref{JLT}).

For a projective hypersurface $X=V(f)\subseteq \mathbb{P}_k^n$ with only isolated singularities, the singular subscheme is defined by the gradient ideal $J(f)$. The proposition~\ref*{ProjectiveLT} gives a criterion for linear type property of the Jacobian ideal. More precisely,  a projective hypersurface is of gradient linear type if and only if at each singular prime,  the affine curve is locally Eulerian. Using this criterion, we prove that any reduced projective curve with simple singularities ($\mathrm{ADE}$) is of gradient linear type (Theorem~\ref{ADE}). There is an example of Jacobian linear type plane curve which has non-simple singularity.  Finally, we show that any reduced projective quartic curve is of gradient linear type. There are examples of quintic and sextic plane curves which are not of Jacobian linear type. 
\section{Jacobian linear type affine hypersurface}
Let $R=k[x_1,\ldots,x_n]$ be a polynomial ring over an algebraically closed  field $k$ of characteristic zero. Let  $f\in R$ be a reduced polynomial defining an affine hypersurface $X=V(f)\subset \AA_k^n$ with only isolated singularities. Let 
\[J(f)=(\dfrac{\partial f}{\partial x_1},\ldots,\dfrac{\partial f}{\partial x_n} )\subset R ,\] 
the so-called  gradient ideal  of $f$. The terminology Jacobian ideal is reserved to the  ideal  $I(f)=(f,J(f))$. The singular locus ${\rm Sing}(X)$ of the hypersurface  $X$ is defined by the Jacobian ideal $I(f)$. Since $X$ has only isolated singularities, it follows that ${\rm Sing}(X)$ consists of finitely many points.

Let $f\in R$ be a reduced polynomial. The { Milnor  algebra}  and the { Tjurina algebra }of $f$ are defined by 
\[M(f):=R/J(f)\quad , \quad T(f):=R/I(f). \]
These algebra are related to the singularities of the corresponding reduced hypersurface $X=V(f)$. The following lemma shows that the local Milnor algebra is Artinian Gorenstein. 
\begin{Lemma}\label{Milnor}
	Let $f\in R$ be a reduced polynomial defining a reduce affine hypersurface $X$ with isolated singularities. Then for each singular point $\fm$  the local  Milnor algebra  $M(f)_{\fm}$ is Artinian Gorenstein ring. 
\end{Lemma}
\begin{proof}
	The localization of $J(f)$ at singular point $\fm$ has codimension $n$ \cite[Lemma 4]{Michler}. Since $R_{\fm}$ is Cohen-Macaulay ring, it follows that the partial derivatives ${\partial f}/{\partial x_i}$ localized at singular point $\fm$ form a $R_{\fm}$-regular sequence. In particular, the local Milnor algebra  $R_{\fm}/J(f)_{\fm}$ is Artinian Gorenstein ring.
\end{proof}

\begin{Remark}\rm
It is well known that a zero dimensional local ring $(A,\fm) $ is Gorenstein if and only if $\hom_{A}(A/\fm,A)\simeq (0:_{A} \fm)$, the socle of $A$, is cyclic. Note that Tjurina algebra of a reduced hypersurface $X\subseteq \mathbb{A}_k^n$  is the coordinate ring of ${\rm Sing}(X)$. In Theorem (\ref{JLT}) we will characterize affine reduced hypersurfaces that the local Tjurina algebra is Artinian Gorenstein ring.  
\end{Remark}

 The numbers 
\[\mu_{\fm}(f):=\dim_k(R_{\fm}/J(f)_{\fm}) \quad ,\quad \tau_{\fm}(f):=\dim_k(R_{\fm}/I(f)_{\fm})\]
are called the {\it Milnor number} and {\it Tjurina number} of the hypersurface $X=V(f)$ at singular point $\fm$, respectively.  Clearly, for each singular point $\fm\in {\rm Sing}(X) $, one has  $\mu_{\fm}(f)\geq \tau_{\fm}(f)$. 
Recall that a hypersurface $X=V(f)\subset \AA_k^n$ with isolated singularities is called {\it locally Eulerian} if $f\in J(f)_{\fm}$ for every singular point $\fm$. Therefore, $X$ is locally Eulerian if and only if  $\mu_{\fm}(f)=\tau_{\fm}(f)$ for each singular point $\fm$. Recall that a polynomial $f\in R$ is \textit{quasi-homogeneous} of degree $d$ and weight $r_i$ if it is satisfies: $f=\sum_{i=1}^{n}(r_i/d)x_i\partial f/\partial x_i$.  
Clearly, any quasi-homogeneous hypersurface is locally Eulerian but the converse in not true.

Let $I=(f_1,\ldots,f_m)$ be an ideal in the commutative Notherian ring $A$. Let $I$ has the minimal presentation
\[A^s\stackrel{\phi}\lar A^m\lar I\lar 0, \] 
where $\phi$ is the first syzygy matrix of the ideal $I$. The { symmetric algebra  } of $I$ is
\[{\rm Sym}_A(I)\simeq A[T_1,\ldots,T_m]/I_1([T_1\ \ldots\ T_m].\phi)\]
where $I_1([T_1\ \ldots\ T_m].\phi)$ is the ideal generated by entries of the multiplication of the  variable matrix $[T_1\ \ldots\ T_m] $ by the matrix $\phi$. Let $\mathcal{R}_A(I):=A[f_1t,\ldots,f_mt]\subset A[t]$ denote the {\rm Rees algebra} of $I$. Then we have a surjective map of $A$-algebras $\psi:{\rm Sym}_A(I)\lar \mathcal{R}_A(I)$ induced  by mapping $T_i\mapsto f_it$, where $f_i$ is viewed in degree 1. The Rees algebra has more complicated presentation ideal. The generators of these ideals referred to, as the defining ideal of the symmetric and Rees algebra of $I$, respectively.  An ideal $I$ is called { of  linear type} if the $A$-algebra homomorphism $\psi$ is injective. For example, if $I$ is generated by a regular sequence, then $I$ is of linear type.  

\begin{Definition}
A reduced affine hypersurface $X=V(f)\subset\AA_k^n$ is  said to be of {\it Jacobian linear type } if the Jacobian ideal $I(f)$ is of linear type.
\end{Definition}
The non-singular and quasi-homogeneous hypersurfaces with only isolated singularities  are of Jacobian linear type.  The following result characterizes Jacobian linear type hypersurfaces only with isolated singularities. 
\begin{Theorem}\label{JLT}
Let $X=V(f)\subset \AA_k^n$ be a reduced affine hypersurface with only isolated singularities. The  following are equivalent.
\begin{enumerate}
	\item[{\rm (a)}] The hypersurface $X$ is of Jacobian linear type. 
	\item[{\rm (b)}] The Jacobian ideal $I(f)$ is locally a complete intersection at singular points. 
	\item[{\rm (c)}] The hypersurface $X$ is locally Eulerian. 
    \item[{\rm (d)}] The Tjurina algebra $T(f)=R/I(f)$ is an  Artinian Gorenstein ring. 
	\item[{\rm (e)}] The $M(f)$-module $\hom_{M(f)}{\left(T(f),M(f)\right)}\simeq \dfrac{J(f):_Rf}{J(f)}$ is cyclic. 
\end{enumerate} 
\end{Theorem}
\begin{proof}
The equivalence of  (a), (b) and (c) was proved in \cite[Theorem 1.1]{Abbas}. Now we prove the equivalence of (c) with (d). Assume that  $X=V(f)\subset \AA_k^n$ is locally Eulerian. Thus $\mu_{\fm}(f)=\tau_{\fm}(f)$ for each singular point $\fm$. Hence the local Tjurina algebra $T(f)_{\fm}$ is equal to local Milnor algebra $M(f)_{\fm}$ and the latter is an Artinian Gorenstein ring by Lemma ~(\ref{Milnor}), which implies the assertion. The ring $T(f)$ is Gorenstein if and only if the local Tjurina algebra $T(f)_{\fm}=R_{\fm}/I(f)_{\fm}$ is  Gorenstein. Since $\dim R_{\fm}=n$ and $\dim T(f)_{\fm}=0$, Krull's prime ideal Theorem implies that 
\[n\leq\mu{(I(f)_{\fm})}\leq n+1, \]  
where $\mu(-)$ denotes the minimal number of generators of an ideal in a local ring. If $\mu{(I(f)_{\fm})}=n$, then the ideal $I(f)_{\fm}$ is a complete intersection, hence a Gorenstein local ring.   Therefore, in view of equivalence (b) with (c), the proof of (d) $\Rightarrow$ (c) will be completed, if we prove that $T(f)_{\fm}$ is not Gorenstein, if $\mu{(I(f)_{\fm})}=n+1$. In this case, $I(f)_{\fm}$ is an almost complete intersection. 
Then \cite[Corollary 1.2]{Kunz} complete the proof. 

Finally, we prove that  (d) $\Leftrightarrow$ (e). Note that $T(f)_{\fm}$ is a finitely generated $M(f)_{\fm}$-module and $M(f)_{\fm}$ is local Artinian   Gorenstein ring. Let $\omega
_{M(f)_{\fm}}$ and $\omega_{T(f)_{\fm}}$
be the canonical modules of $M(f)_{\fm}$ and $T(f)_{\fm}$, respectively.  For each singular point $\fm$ we have 
\begin{eqnarray}
\nonumber \omega_{T(f)_{\fm}}&\cong& \hom_{M(f)_{\fm}}(T(f)_{\fm},\omega_{T(f)_{\fm}}),\\
\nonumber &\cong &\hom_{M(f)_{\fm}}(T(f)_{\fm},M(f)_{\fm}),\\
\nonumber &\cong & (J(f):_Rf)_{\fm}/J(f)_{\fm},
\end{eqnarray}
where the first isomorphism follow by \cite[Proposition 21.4]{DE} and the second because $M(f)_{\fm}$ is Gorenstein. By \cite[Proposition 21.5]{DE} $T(f)_{\fm}$ is Gorenstein if and only if  $\omega_{T(f)_{\fm}}$ is a cyclic module. 
\end{proof}

\begin{Corollary}
	Let $X=V(f)\subset \AA_k^n$ be a locally Eulerian affine hypersurface with only isolated singularities. Then the Rees algebra  of $I(f)\subset R$ and the symmetric algebra of $I(f)/(f)\subset R/(f)$ are Cohen-Macaulay.   
\end{Corollary}
\begin{proof}
Since $X$ is locally Eulerian, it follows that $\hht(I(f)_{\fm})=\mu(I(f)_{\fm})=n$, $\hht((I(f)/(f))_{\fm/(f)})=n-1$ and $\mu((I(f)/(f))_{\fm/(f)})=n$ for any singular point $\fm$. Thus we apply the criterion~\cite[Theorem 10.1]{Trento}. 
\end{proof}
A quasi-homogeneous polynomial $f\in R=k[x_1,\ldots,x_n]$ is said to be {\it non-degenerate} if the origin of $\AA_k^n$ is an isolated singular point of $X=V(f)\subseteq \AA_k^n$. We say that a polynomial has {\it order} $d$, if all of its monomials have degree $d$ or higher. Denote by $\nu(f)$ the order of $f$. 

\begin{Definition}
	A polynomial $f\in R=k[x_1,\ldots,x_n]$ is said to be {\it semiquasi-homogeneous } of degree $d$ with weight $r_i$ if it is of the form $f=F+G$, where $F$ is a non-degenerate quasi-homogeneous polynomial of degree $d$ with weight $r_i$ and  G is a polynomial of order $\nu(G)=(d/\nu(f))+1$. 
\end{Definition}
We say that a reduced hypersurface $X\subset\AA_k^n$ is semiquasi-homogeneous hypersurface if $X$ is defined by a semiquasi-homogeneous polynomial. Let $X=V(f)\subset \AA_k^2$ be a reduced curve define by the semiquasi-homogeneous polynomial $f=x^5-y^6 + x^3y^4$ with $F= x^5+y^6$ and $G=x^3y^4$.  Then $X$ is  semiquasi-homogeneous reduce curve with only one singular point at origin. By \cite[Example 2.6]{Abbas} $X$ is not locally Eulerian and hence is not of Jacobian linear type. A computation in \cite{singular} yields that the generators of $I(f)$ are analytically dependent. In the following result we introduce families of  semiquasi-homogeneous hypersurface of Jacobian linear type.

\begin{Proposition}\label{semiquasi}
	Let $X=V(f)\subset \AA_k^2$ be a  reduced  curve defined by a semiquasi-homogeneous $f=F+G$ of degree $d$ with weight $r_1,r_2$. Let $k\geq 1$.  If $F$ is one of the following non-degenerate quasi homogeneous polynomials 
	\[ y^2x-x^{k+2}  \ , \  y^3-x^{3k+1}  \ , \  y^3-x^{3k+2} \ , \ y^2-x^{k+1} \]
	then $X$ is of Jacobian linear type.   
\end{Proposition}
\begin{proof}
	Let $f=y^2x-x^{k+2}+G(x,y)$, where $G$ has initial degree at least $k+3$. The gradient  ideal is  generated by
	\[ f_x=-(k+2)x^{k+1}+y^2+ G_ x \quad , \quad   f_y=2yx+G_y. \]
	We can write 
	\begin{eqnarray}
	\nonumber G_x&=&x^{k+1}h_0(x,y)+x^{k}y^2h_1(y)+\ldots+xy^{k+1}h_{k}(y)+y^{k+2}h_{k+1}(y),\\
	\nonumber G_y&=&xyg_0(x,y)+x^{k+2}g_1(x)+y^{k+2}g_2(y). 
	\end{eqnarray}
	where $h_0(0,0)=0$, the initial degree of $h_i$ is at least $0$, the initial degree of $g_0$ is at least $k$ and the initial degree of $g_1, g_2$ is at least $0$. One has
	\begin{eqnarray}
	\nonumber f_x&=&x^{k+1}(h_0(x,y)-(k+2))+y^2(1+x^{k}h_1+\ldots +y^{k}h_{k+1}),\\
	\nonumber f_y&=&xy(g_0(x,y)+2)+x^{k+2}g_1(x)+y^{k+2}g_2(y). 
	\end{eqnarray}
	Since the polynomials  $h_0(x,y)-(k+2), g_0(x,y)+2 $ are unit locally at $\fm=(x,y)$, it follows that the  gradient ideal $J(f)$  locally at $\fm$ is generated by 
\[H_1:=x^{k+1}-\alpha y^2P(x,y) \quad,\quad H_2:=xy-\beta (x^{k+2}g_1(x)+y^{k+2}g_2(y)), \]
with $\alpha,\beta $ units and where $P(x,y)= 1+x^{k}h_1+\ldots +y^{k}h_{k+1}$. We claim that $(x^{k+2},y^3,xy)_{\fm}\subseteq J(f)_{\fm}$, which proves that $X=V(f)$ is of Jacobian linear type by the Theorem~(\ref{JLT}). 
The equation 
	\begin{eqnarray}
	\nonumber (-y+\beta g_1x^{k+1}(1+\alpha \beta yPg_1))H_1 +x^k(1+\alpha \beta yPg_1)H_2 =y^3(\alpha P - \beta x^kU(x,y))
	\end{eqnarray}
shows that $y^3\in J(f)_{\fm}$, where $U(x,y)= \alpha ^2\beta xP^2g_1^2+g_2y^{k-1}(1+\alpha \beta yPg_1)$.
Denote by $H_3$ the reminder of the polynomial $H_2$ with respect to the monomial $y^3$. One has 
	\[H_3:=xy-\beta x^{k+2}g_1(x). \] 
We have
\[\beta g_1xH_1 +(1+\alpha \beta yPg_1)H_3 =xy(1-\alpha \beta ^2 Pg_1^2x^{k+1}),\]
which shows that  $xy\in J(f)_{\fm}$. Also using the polynomial $xH_1$, we conclude that  $x^{k+2}\in J(f)_{\fm}$, which complete the proof of the claim. 

\medskip

Let $f=y^3-x^{3k+1}+G(x,y)$, where $G$ has initial degree at least $3k+2$. The gradient  ideal is generated by

\[f_x=-(3k+1)x^{3k}+G_x \quad , \quad  f_y=3y^{2}+G_y. \]

We can write 
\begin{eqnarray}
\nonumber G_x&=&x^{3k}h_1(x,y)+x^{3k-1}y^2h_2(y)+\ldots +xy^{3k}h_{3k}(y)+y^{3k+1}h_{3k+1}(y),\\
\nonumber G_y&=&y^2g_0(x,y)+yx^{3k}g_1(x)+x^{3k+1}g_2(x). 
\end{eqnarray}
where $h_1(0,0)=0, g_0(0,0)=0$, the initial degree of $h_i$ is at least $0$, the initial degree of $g_0$ is at least $3k-1$ and the initial degree of $g_1, g_2$ is at least $0$. One has
\begin{eqnarray}
\nonumber f_x&=&x^{3k}(h_1(x,y)-(3k+1))+x^{3k-1}y^2h_2(y)+\ldots +xy^{3k}h_{3k}(y)+y^{3k+1}h_{3k+1}(y),\\
\nonumber f_y&=&y^2(g_0(x,y)+3)+yx^{3k}g_1(x)+x^{3k+1}g_2(x). 
\end{eqnarray}
Since $h_1(x,y)-(3k+1), g_0(x,y)+3 \not\in \fm $, it follows that the the gradient ideal $J(f)$  locally at $\fm$ is generated by 
\[H_1:=x^{3k}-\alpha y^2P(x,y)\quad,\quad H_2:=y^2-\beta x^{3k}(yg_1(x)+xg_2(x)), \]
where $P(x,y)= x^{3k-1}h_2(y)+\ldots +xy^{3k-2}h_{3k}(y)+y^{3k-1}h_{3k+1}(y)$. 
We have
\begin{center}
	$H_1 +\alpha P(x,y) H_2 = x^{3k}(1-\alpha \beta (yg_1+xg_2)P(x,y)),$
\end{center}
which shows that $J(f)_{\fm}=(x^{3k}, y^2)_{\fm}$. Therefore, $X=V(f)$ is locally Eulerian.

A similar argument as above shows that the semiquasi-homogeneous hypersurfaces defining by the polynomials $f_1=y^2-x^{k+1}$ and $ f_2=y^3-x^{3k+2}$  are of Jacobian linear type. In fact, 
\[J(f_1)_{\fm}=(x^k,y)_{\fm} \quad,\quad J(f_2)_{\fm}=(x^{3k+1},y^2)_{\fm}. \]
\end{proof}
It is natural to ask whether a reduced affine hypersurface with higher dimensional singularities is of Jacobian linear type. We have seen that a quasi-homogeneous hypersurface with isolated singularities is of gradient linear type while for higher dimensional singularities it is not true in general. For instance, consider the polynomial $f=x^3y^2+x^5z+y^4$, which defines  a quasi-homogeneous hypersurface with weight $(2,3,2)$. It is easy to check that  $\dim {\rm Sing}(X)=1$. The Jacobian ideal $I(f)$ localized at singular prime $\fq=(x,y)$ is not a complete intersection. A computation in \cite{singular} yields that the Jacobian ideal is not of linear type. Therefore, for higher dimensional singularities it is not easy to characterize Jacobian linear type hypersurface.

Let $X=V(f)\subseteq \mathbb{A}_k^n$ be an affine hypersurface with higher dimensional singularities. Let $\fq$ be a minimal prime of the Jacobian ideal $I(f)$. Then Milnor and Tjurina algebra of $f$ locally at $\fq$ are Artinian local ring and hence have finite length. More precisely, 
\[\ell (T(f)_{\fq})\leq \ell (M(f)_{\fq}), \]
where $\ell(-)$ denote the length. It looks to pose the following for one  dimensional singularity. 
\begin{Conjecure}
Let $X=V(f)\subseteq \mathbb{A}_k^n$ be an affine hypersurface with one dimensional singularities. The hypersurface $X$ is of Jacobian linear type if and only if for each minimal prime $\fq$ of $I(f)$, $ \ell (T(f)_{\fq})= \ell (M(f)_{\fq})$ and $I(f)_{\fq}$ is a complete intersection.  
\end{Conjecure}

\section{Gradient linear type projective hypersurface}
Let $X=V(f)\subseteq \pp_k^n$ be a reduced projective hypersurface  defining by a reduced homogeneous polynomial $f\in R=k[x_0,\ldots,x_n]$ of degree $d\geq 3$. By the Euler formula, the singular subscheme of $X$ is defined by the gradient ideal 
$$J(f)=(\partial f/ \partial x_0, \ldots,\partial f/ \partial x_n) .$$ 

\begin{Definition}
A reduced projective hypersurface $X=V(f)$ is said to be of {\it gradient linear type} if the gradient ideal $J(f)\subseteq R$ is of linear type.   
\end{Definition}
It is well-known that if $X$ is non-singular, then $J(f)$ is generated by a regular sequence and hence is of gradient linear type.  The following question arises.
\begin{Question}\label{GLT-hypersurface}
	Which projective singular  hypersurfaces are of gradient linear type? 
\end{Question}
For hypersurfaces with isolated singularities, we construct a criterion for linear type property of Jacobian ideal. Assume that $X$ has isolated singularities. Then ${\rm Sing}(X)$ consists of finitely many projective points. Let $p\in \pp_k^n$ be a singular point of $X$. By a projective transformation, we may assume that $p=[0:0:\cdots:0:1]$. The ideal of the points $p$ is the prime ideal $\fp=(x_0,\ldots,x_{n-1})$. Consider the affine chart $U_{x_n}=\AA_k^n$ with coordinate ring $A=k[x_0/x_n,\ldots,x_{n-1}/x_n]=k[T_1,\ldots,T_n]$. The equation of $f$ in this affine chart is $F(T_1,\ldots,T_n)=f(x_0,\ldots,x_{n-1},1)$. 
 By  \cite[Lemma 3.1 and Corollary 3.2]{AA} and \cite[Corollary 3.2]{Abbas}, we have the following characterization for gradient linear type hypersurfaces with only isolated singularities:

\begin{Proposition}\label{ProjectiveLT}
Let $X=V(f)\subset \pp_k^n$ be a reduced projective hypersurface with only isolated singularities. The following are equivalent:
	\begin{enumerate}
		\item[\textrm{(a)}] The hypersurface  $X$ is of gradient linear type. 
		\item[\textrm{(b)}] Locally at each singular prime the gradient ideal $J(f)$ is a complete intersection. 
		\item[\textrm{(c)}] If $R^m\stackrel{\varphi}\lar R^{n+1}\lar J(f)\lar 0$ is the minimal presentation of $J(f)$, then the ideal generated by the entries of $\varphi$ has codimension $n+1$. 
		\item[\textrm{(d)}] For each singular prime  $\fp$ such that $x_i\notin \fp$, the Jacobian ideal $I(F)$ in the affine coordinate ring $A=k[x_0/x_i,\ldots,x_n/x_i]$ is a locally complete intersection at $\fq$ where $F=f(x_1,\ldots,x_{n-1},1)$ and $\fq\subset A$ is the maximal ideal associated to $\fp$. 
	\end{enumerate} 
\end{Proposition}
\begin{Example}
Let $X\subseteq \pp_k^3$ be a quartic surface defining by the polynomial 
\[f=x^4+y^4+z^4+z^2w^2+xyzw.\]
The surface $X$ has one singular point $[0:0:0:1]$. 
It is easy to check that the vector 
\[(-8y^3z\ ,\  2yz^2w\ ,\ 63zw^3\ ,\ 32x^2y^2-64xyzw+126z^2w^2+63w^4)\]
is a syzygy of the gradient ideal $J(f)$. Some slightly painful but straightforward calculation yields that the ideal generated by Kosuzl syzygies and above syzygy is an $(x,y,z,w)$-primary ideal.  Therefore, Proposition~(\ref{ProjectiveLT})(c) shows that $X$ is of gradient linear type.  
\end{Example}
\begin{Example}
Consider the cubic surface $X=V(f)$ in $\pp_k^3$ defining by the polynomial $f=xzw+x^2y+y^2z-z^3$. The surface $X$ has one singular point $p=[0:0:0:1]$. Denote by $\fm_p=(x,y,z)$ the prime ideal of the point $p$. It is easy to check that $(24yz+5)F\in J(F)_{\fm_p}$, where $F$ is the affine equation of $f$ in the affine coordinate ring $k[x,y,z]$. Since $(24yz+5)\not\in \fm_p$, it follows that the affine curve $V(F)$ is locally Eulerian. Then Proposition~(\ref{ProjectiveLT})(d) proves that the cubic surface $X$ is of gradient linear type.  
\end{Example}
For hypersurfaces with higher dimensional singularities, the question(~\ref{GLT-hypersurface}) will be very difficult. The following examples illustrate the difficulty of the problem.  
\begin{Example}
	\begin{enumerate}
\item[{\rm (a)}] The surface $X=V(f)\subseteq \mathbb{P}_k^3$ defining by the polynomial $f=x^4-xyw^2+zw^3$ is one dimensional singularity. The only minimal prime of $J(f)$  is $\fq=(x,w)$. The gradient ideal locally at $\fq$ is not a complete intersection while the following vectors are the syzygies of $J(f)$
\[(0,2y, 3z,-w)\quad,\quad  (0,w,x,0),\]
which implies that the ideal generated by entries of the first syzygy  matrix of $J(f)$ has codimension $4$. A computation in \cite{singular} yields that the defining ideal of the Rees algebra contains the following polynomial:
\[4xT_2^2-wT_1T_3-yT_3^2. \]
\item[{\rm (b)}]  Let $X\subseteq \pp^3$ be a  reduced surface of degree $4$ defining by the polynomial 
		\[f=x^2z^2+x^2w^2+y^2z^2+z^2w^2.\]
		Then $\mathrm{Sing}(X)$ consists of the two lines $\fp=(x,z)$, $\fq=(z,w)$ and the isolated singular point $\fm=(x,y,w) $. Thus  $\dim \mathrm{Sing} X=1$. The gradient ideal $J(f)$  is a complete intersection locally at minimal primes. In fact, we have 
		$$J(f)_{\fp}=\fp_{\fp}\quad, \quad J(f)_{\fq}=\fq_{\fq} \quad,\quad J(f)_{\fm}=\fm_{\fm}.  $$  
The ideal generated by entries of the first syzygy matrix of $J(f)$ has codimension $4$ and $X$ is of gradient linear type. 
	\end{enumerate}
	
\end{Example}

\subsection{Gradient linear type plane curves}
Let $X=V(f)\subseteq \pp^2$ be a reduced plane curve defining by the homogeneous polynomial $f\in R=k[x,y,z]$ of degree $d\geq 3$.  
Since $X$ is reduced, the singular locus $\mathrm{Sing}(X)$  consists of finitely many points in $\pp_k^2$.  Denote by $\mathrm{m}_p(f)$ the multiplicity of  the singular point at $p\in X$. By a projective transformation, we may assume that $p=[0:0:1]$. Consider the affine chart $U_z=\mathbb{A}_k^2$.  Keeping the same notation, the equation of $f$ in this affine chart is $F(x,y):=f(x,y,1)$.   

Let $p\in\mathrm{Sing}(X)$ be a singular point with $\mathrm{m}_p(f)=2$. The affine curve $X_z:=V(F(x,y))$ is an analytic subset of $\mathbb{A}_k^2$ and $(X_z,(0,0))$ is an analytic set germ. If this set germ is analytically equivalent to the analytic set germ $(V(y^2-x^{k+1}),(0,0))$ for some $k\in \mathbb{N}$, then $p$ is called an $\mathrm{A}_k$ singularity of $X$.  We call the singularity $\mathrm{A}_1$ a \textit{node}, the singularity $\mathrm{A}_2$ a \textit{ simple cusp}, the singularity $\mathrm{A}_3$ a \textit{Tacnode}, the singularity $\mathrm{A}_4$ a \textit{Ramphoid cusp} and the singularity $\mathrm{A}_5$ an \textit{Oscnode}. If $p\in X$ is $\mathrm{A}_k$ singularity  for $k\geq 2$, then the intersection  multiplicity $\mathrm{mult}_p(X,L)=k$, where $L$ is the tangent line at $p$. 

Now assume that $\mathrm{m}_p(f)=3$. If the set germ $(X_z,(0,0))$ is  analytically equivalent to the analytic set germ $(V(y^2x-x^{k+2}),(0,0))$ for $k\geq 1$, then $p$ is called $\mathrm{D}_{k+3}$ singularity.  We call the singularity $\mathrm{D}_4$ a \textit{ordinary triple point} and the singularity $\mathrm{D}_5$ a \textit{Tacnode cusp}. 
If  the set germ $(X_z,(0,0))$ is  analytically equivalent to the analytic set germ $(V(x^{3}-y^4),(0,0))$, then $p$ is called  $\mathrm{E}_6$(\textit{multiplicity $3$ cusp}) singularity. If the set germ $(X_z,(0,0))$ is  analytically equivalent to the analytic set germ $(V(y^{3}-yx^3),(0,0))$, then $p$ is called   $\mathrm{E}_7$ singularity. 
Finally, if  the set germ $(X_z,(0,0))$  analytically equivalent to $(V(x^{3}-y^5),(0,0))$, then $p$ is called   $\mathrm{E}_8$ singularity.  
The singularities $\mathrm{A}_k, \mathrm{D}_k, \mathrm{E}_6, \mathrm{E}_7$ and $\mathrm{E}_8$ are called simple singularities.  

\begin{Theorem}\label{ADE}
Any reduced projective  plane curve with simple singularities is of gradient linear type. 
\end{Theorem}
\begin{proof}
Let $X=V(f)\subseteq\pp^2$ be a reduced plane curve defined by the reduced polynomial $f\in R=k[x,y,z]$ of degree $d\geq 3$. Let $p\in \mathbb{P}_k^2$ be a singular point of $X$. Denote by $\fq\subset R$, the prime ideal of the point $p$. By a projective transformation, we may assume that $p=[0:0:1]$ and hence  $\fq=(x,y)$. For  each singular type we prove that the affine equation of $X$ is locally Eulerian and hence of Jacobian linear type. Therefore, Proposition~\ref{ProjectiveLT}(d) completes the proof. 

\medskip

\underline{$\mathrm{A}_k$ singularity:}  The affine equation of $X$ is 
\[F(x,y)=y^2-x^{k+1}+ax^ky+y^2u(x,y)+h(x,y),\]
where $u(x,y)$ is a homogeneous polynomial of degree $k-1$ and  $h(x,y)$ has initial degree $k+2$. The gradient  ideal of $F$ is generated by
\[F_x=-(k+1)x^k+kax^{k-1}y+y^2u_x+h_x\quad , \quad  F_y=2y+ax^k+2yu+ y^2u_y+h_y.  \]
We can write 
\begin{eqnarray}
\nonumber h_x&=&x^{k}h_0(x,y)+x^{k-1}H_1(y)+\ldots +xH_{k-1}(y)+H_k(y),\\
\nonumber h_y&=&yg_0(x,y)+G_1(x),
\end{eqnarray}
where the initial degree of $H_i$ is at least $i+1$, the initial degree of $h_0,g_0$ is at least $1$ and  the initial degree of $G_1(x)$ is at least $k+1$. By the condition, We can write $H_i(y)=yh_i(y)$ and $ G_1(x)=x^kg_1(x)$. One has
\begin{eqnarray}
\nonumber F_x&=& x^k(-(k+1)+h_0(x,y))+yP(x,y),\\
\nonumber F_y&=&y(2+g_0(x,y)+2u+yu_y)+x^k(a+g_1(x)),
\end{eqnarray}
where $P(x,y) = kax^{k-1} +yu_x+x^{k-1}h_1+\ldots + h_k$. Note that when $k = 1$; we have $u(x,y) = 0$. Since the polynomials $h_0(x,y)-(k+1)$ and $2+g_0(x,y)+2u+yu_y$ are unit locally at $\fq=(x,y)$,  it follows that the gradient ideal $J(F)$ is generated by the following polynomials locally at $\fq$
\begin{eqnarray}
\nonumber  T_1& =&x^k+\alpha yP(x,y),\\
\nonumber  T_2 &=&y+\beta x^kQ(x,y),
\end{eqnarray}
with $\alpha,\beta $ unit and $Q(x,y) = a+g_1(x)$. We claim that  $J(F)_{\fq}=(x^k,y)_{\fq}$, which proves that $V(F)$ is locally Eulerian. One inclusion is clear by conditions on initial degrees. 
We get
\begin{eqnarray}
\nonumber T_2-\beta QT_1 &= &y(1-\alpha \beta PQ)\\
\nonumber  T_1-\alpha PT_2 &=& x^k(1-\alpha \beta PQ),
\end{eqnarray}
which proves  that $(x^k,y)\in J(F)_{\fq}$. 
\medskip

\underline{$\mathrm{D}_{k+3}$ singularity:} For $k=1$ the affine equation of $X$ is $F(x,y)=y^2x-x^3+h(x,y)$, where $h(x,y)$ has initial degree at least $4$. The polynomial $F$ is semiquasi-homogeneous. Hence Proposition~\ref{semiquasi} shows that $V(F)$ is of Jacobian linear type. 
For $k\geq 2$ the affine equation of $X$ is 
\[F(x,y)=y^2x-x^{k+2}+yu(x,y)+h(x,y),\]
where $u(x,y)$ is a homogeneous polynomial of degree $k+1$ and $h(x,y)$ has initial degree at least $k+3$. Setting $E(x,y):=yu(x,y)$.  The gradient ideal of $F$ is generated by 
\[F_x=-(k+2)x^{(k+1)}+y^2+E_x+h_x \quad , \quad  F_y=2yx+E_y+h_y. \]
We can write 
\begin{eqnarray}
\nonumber h_x&=&x^{k+1}h_0(x,y)+x^{k}H_1(y)+\ldots +xH_{k}(y)+H_{k+1}(y),\\
\nonumber h_y&=&xyg_0(x,y)+G_1(x)+G_2(y). \\
\nonumber E_x&=&a_1x^{k}y+\ldots +a_{k}xy^{k}+a_{k+1}y^{k+1},\\
\nonumber E_y&=&xym_0(x,y)+m_1x^{k+1}+m_y^{k+1}2, 
\end{eqnarray}
 where  $m_0(0,0)=h_0(0,0)=0$, the initial degree of $H_i(y)$ is at least $i+1$, the initial degree of $g_0$ is at least $k$ and the initial degree of $G_1(x),G_2(y)$ is at least $k+2$. By condition on initial degree, we can write 
 \[H_i(y) = y^{i+1}h_i(y)\quad,\quad G_1(x) = x^{k+2}g_1(x) \quad,\quad G_2(y)= y^{k+2}g_2(y). \]
Setting $D_i=yh_i+a_i\ , \ E_1 = xg_1+m_1\ ,\ E_2=yg_2+m_2$. We have 
\begin{eqnarray}
\nonumber F_x&=&x^{k+1}(h_0(x,y)-(k+2))+x^{k}yD_1(y)+\ldots+xy^{k}D_{k}(y)+y^{k+1}D_{k+1}(y)+y^2,\\
\nonumber F_y&=&xy(g_0(x,y)+m_0(x,y)+2)+x^{k+1}E_1(x)+y^{k+1}E_2(y). 
\end{eqnarray}
Since $h_0(x,y)-(k+2), g_0(x,y)+m_0(x,y)+2 $ are unit locally the ideal $\fq=(x,y)$, it follows that the the gradient ideal $J(F)$  locally at $\fq$ is generated by the polynomials 
\begin{eqnarray}
 \nonumber T_1 &:= &x^{k+1}-\alpha(x^{k}yD_1(y)+y^2P(x,y)),\\
 \nonumber T_2 &:= &xy-\beta (x^{k+1}E_1(x)+y^{k+1}E_2(y)),
\end{eqnarray}
where $P(x,y) = 1+y^{k-1}D_{k+1}(y)+\ldots+ x^{k-1}D_2(y)$ and  $\alpha,\beta$ are units. We claim that $(x^{k+2},y^3,xy^2,x^2y)\subseteq J(F)$ in the ring $k[x,y]_{\fq}$. 
We get $\lambda_1T_1+\lambda_2T_2=y^3(\alpha P+\mathcal{U}(x,y))$, where 
\begin{eqnarray}
\nonumber  \lambda_1 & = & -y+\beta E_1x^{k-2}(x^2+\alpha \beta PE_1xy+\alpha ^2\beta PE_1y^2(\beta PE_1+D_1)),\\
 \nonumber  \lambda_2 &=& x^{k-2}(x^2+\alpha( \beta PE_1- D_1)xy+\alpha ^2\beta^2P^2E_1^2y^2),
\end{eqnarray}
and $\mathcal{U}(x,y)$ is a polynomial in $k[x,y]_{\fq}$. Hence locally at $\fq$ the monomial $y^3$ belongs to the ideal $J(F)$. Denote by $T_3$ the reminder of the polynomial $T_2$ to the monomial $y^3$. One has 
\[  T_3:=xy-\beta x^{k+1}E_1(x). \] 
The following relations show that the monomials $x^2y$ belong to the ideal $J(F)$ locally at $\fq$. 
\begin{eqnarray}
\nonumber   \beta xE_1T_1+(x+\alpha \beta PE_1y)T_3 &=x^2y(1-\alpha \beta x^{k-1}E_1(D_1+\beta PE_1)),  
\end{eqnarray} 
 Finally, using the polynomials  $yT_3,xT_1$, we conclude that $xy^2,x^{k+2}\in J(F)_{\fq}$ which complete the proof of the claim. Therefore, since $F\in(x^{k+2},y^3,xy^2,x^2y)\subseteq J(F)_{\fq}$, it follows that $F$ is locally Eulerian for $k\geq 2$.
\\

 \medskip

 \underline{$\mathrm{E}_6$ singularity:}
  The affine equation of $X$ is 
 \[ F(x,y)=x^3+y^4+ax^3y+bx^2y^2+cx^4+h(x,y),\]
 where  $h(x,y)$ has initial degree at least $5$. 
The gradient ideal of $F$ is generated by 
\[F_x =3x^2+3ax^2y+2bxy^2+4cx^3+h_x\quad ,\quad  F_y=4y^3+ax^3+2bx^2y+h_y.\]
We can write 
\begin{eqnarray}
\nonumber h_x&=&x^2h_1(x,y)+xH_2(y)+H_3(y),\\
\nonumber  h_y&=&y^3g_1(x,y)+y^2G_2(x)+yG_3(x)+G_4(x),
\end{eqnarray}
where $g_1(0,0)=h_1(0,0)=0$ and  the initial degree of $H_i(y)$ and $G_i(x)$ is at least $i+1$ and $i$, respectively. By conditions on the initial degree, we have  $H_i(y)=y^{i+1}h_i(y), G_i(x)=x^ig_i(x)$. Thus
\begin{eqnarray}
\nonumber F_x &=& x^2(3+h_1+3ay+4cx)+2bxy^2+xy^3h_2+y^4h_3,\\
\nonumber F_y &=& y^3(g_1+4)+ax^3+2bx^2y+y^2x^2g_2+yx^3g_3+x^4g_4.
\end{eqnarray}
Since the polynomials  $3+h_1+3ay+4cx$ and $g_1+4$ are unite locally at $\fq$,  it follows that the gradient ideal $J(F)$ is generated by the following polynomials locally at $\fq$
\begin{eqnarray}
\nonumber  T_1& =&x^2+\alpha y^2P(x,y),\\
\nonumber  T_2 &=&y^3+\beta x^2Q(x,y),
\end{eqnarray}
where $P(x,y)=2bx+xyh_2+y^2h_3$, $Q(x,y)=ax+2by+y^2g_2+yxg_3+x^2g_4$ and $\alpha,\beta$ are units. The following relation shows that $y^3\in J(F)_{\fq}$. 
\[T_2-\beta(Q(x,y)-2 \alpha by^2(a+g_4x)) T_1 = y^3(1+M(x,y)),\]
where  $M(x,y)\in k[x,y]_{\fq}$ is a polynomial of degree at least one. 
Denote by $T_3$ the reminder of $T_1$ to the monomial $y^3$. We get 
\[ T_3:= x^2+2\alpha bxy^2.\]
We consider two cases. If $b=0$, then $x^2\in J(F)_{\fq}$. If $b\neq 0$, then 
\[(2bx+g_2xy)T_3-2b\alpha yx^2Q(x,y) = x^3(2b +(g_2-2ab\alpha )y-2bg_4\alpha xy-2bg_3\alpha y^2),\]
which proves that $x^3\in J(F)_{\fq}$. Using the polynomial $x^2Q(x,y)$, we conclude that $x^2y$ belongs to $J(F)_{\fq}$. Since $f\in (y^3,x^3,x^2y)\subseteq J(F)_{\fq}$, it follows that $F$ is locally Eulerian. 

\medskip
 
 \underline{$\mathrm{E}_7$ singularity:} The affine equation of $X$ is 
 \[F(x,y)=y^3-yx^{3}+bx^2y^2+cxy^3+dy^4+G(x,y), \]
 where $G$ has initial degree at least $5$. The gradient  ideal  of $F$ is generated by
 \[F_x=-3yx^2+2bxy^2+cy^3+G_x \ , \ F_y=3y^2-x^3+2bx^2y+3cxy^2+4dy^3+G_y. \]
  We can write 
 \begin{eqnarray}
 \nonumber G_x&=&x^4h_1(x)+yx^{2}h_2(x,y)+xy^3h_3(y)+y^4h_4(y),\\
 \nonumber G_y&=&y^2g_0(x,y)+yx^3g_1(x)+x^4g_2(x),
 \end{eqnarray}
 where $h_0(0,0)=0$, the initial degree of $h_i$ is at least 0, the initial degree of $g_0$ is at least $k+1$ and the initial degree of $g_1,g_2$ is at least 0. We have 
\begin{eqnarray}
 \nonumber F_x&=&yx^2(-3+h_2(x,y))+x^4h_1(x)+2bxy^2+y^3P(x,y),\\
  \nonumber F_y &=& x^3(-1+yg_1(x)+xg_2(x))+ y^2Q(x,y)+2bx^2y.
\end{eqnarray}
where $P(x,y):= c+xh_3(y)+yh_4(y)$ and $Q(x,y):= 3+g_0(x,y)+3cx+4dy$. 
Since the polynomials $h_2(x,y)-3$ and $-1+yg_1(x)+xg_2(x)$ are  units locally at $\fq$, the gradient ideal is generated by 
 \[T_1:=yx^2-\alpha (x^4h_1(x)+2bxy^2+y^3P(x,y))\ , \  T_2:=x^3-\beta (2bx^2y+y^2Q(x,y)). \]
 We have $M_1T_1+M_2T_2 = y^3(\beta Q(x,y)+U(x,y))$, where 
 \begin{eqnarray}
 \nonumber  M_1 &= &x+(2\alpha b+\alpha \beta Qh_1-2b\beta )y,\\
 \nonumber  M_2 & =& -y+\alpha h_1x^2+\alpha ^2 h_1(2b+\beta Qh_1)xy+2\alpha ^2b\beta h_1(2b+\beta Qh_1)y^2,
  \end{eqnarray}
  and $U(x,y)\in k[x,y]_{\fq}$ is a polynomial with initial degree at least 1. Since the polynomials $Q(x,y)\not\in \fq$, it follows that $y^3\in J(F)_{\fq}$. 
  Denote by $T_3$ the reminder of $T_1$ to the monomial $y^3$. We have 
  \[N_1T_3+N_2T_2 = x^3y(\beta Q + W(x,y)),\]
  where 
  \[N_1=\beta xQ+\alpha\beta Q(2b+\beta Qh_1)y \ , \ N_2 =(\alpha \beta Qh_1)x^2-2\alpha ^2b(2b+\beta Qh_1)xy,\]
 and  $W(x,y)$ has the initial degree at least 1. Then $x^3y\in J(F)_{\fq}$. 
Assume that $b\neq 0$. Using $yT_2, x^2T_2$, we assert that $x^2y^2,x^5\in J(F)_{\fq}$. Then $f\in (y^3,x^3y,x^2y^2,x^5)\subseteq J(F)_{\fq}$. If $b=0$, then $x^5\in J(F)_{\fq}$ because of $xT_1$. Then $f\in (y^3,x^3y,x^5)\subseteq J(F)_{\fq}$, Therefore, in these cases $F$ is locally Eulerian. 
  
\medskip

\underline{$\mathrm{E}_8$ singularity:}
The affine equation of $X$ is 
\[F(x,y)=x^3+y^5+ax^5+bx^4y+cx^3y^2+dx^2y^3+h(x,y),\]
where  $h(x,y)$ has initial degree at least $6$.  
The gradient ideal of $F$ is generated by 
\[F_x =3x^2+5ax^4+4bx^3y+3cx^2y^2+2dxy^3+h_x\quad ,\quad  F_y=5y^4+bx^4+2cx^3y+3dx^2y^2+h_y.\]
We can write 
\begin{eqnarray}
\nonumber h_x&=&x^2h_1(x,y)+xH_2(y)+H_3(y),\\
\nonumber h_y&=&y^4g_1(x,y)+y^3G_2(x)+y^2G_3(x)+yG_4(x)+G_5(x),
\end{eqnarray}
where $g_1(0,0)=h_1(0,0)=0$.  We can write  $G_i(x) = x^ig_i(x)$ and $H_2(y)= y^4h_2(y), H_3(y)=y^5h_3(y)$. Using these relation, we get 
\begin{eqnarray}
\nonumber F_x &=& x^2(3+h_1+5ax^2+4bxy+3cy^2)+2dxy^3+xy^4h_2+y^5h_3\\
\nonumber F_y &=& y^4(g_1+5)+x^2(bx^2+2cxy+3dy^2+y^3g_2+y^2xg_3+yx^2g_4+x^3g_5).
\end{eqnarray}
Note that $(3+h_1+5ax^2+4bxy+3cy^2)$ and $(g_1+5)$ are unite locally at $\fq$.  We have the following generators for gradient ideal
\begin{eqnarray}
\nonumber  T_1& =&x^2+\alpha y^3P(x,y),\\
\nonumber  T_2 &=&y^4+\beta x^2Q(x,y),
\end{eqnarray}
where $P(x,y) = 2dx+xyh_2+y^2h_3$ and $Q(x,y) = bx^2+2cxy+3dy^2+y^3g_2+y^2xg_3+yx^2g_4+x^3g_5$ . The following equation proves that  $y^4\in J(F)_{\fq}$.
\[T_2+\beta (2\alpha dxy^3(b+xg_5)-Q)T_1 = y^4(1+M(x,y)),\]
where $M(x,y)\in k[x,y]_{\fq}$ has no constant term. Denote by $T_3$ the reminder of $T_1$ to the monomial $y^4$ and  by $T_4$ the reminder of $x^2Q(x,y)$ to the monomial $x^3$. We get 
\[T_3 = x^2+2d\alpha xy^3 \quad, \quad T_4 = x^2y^2(3d+yg_2(x))  \] and 
\[(3d+g_2y)xT_3-2\alpha dyT_4 = x^3(3d+g_2y).\]
When $ d\neq 0$ , we assert that $x^3,x^2y^2 \in J(F)_{\fq}$ therefore ,$f\in (y^4,x^3,x^2y^2) \subseteq J(F)_{\fq}$  and when $ d = 0$, $x^2  \in J(F)_{\fq}$ and $f\in (y^4,x^2) \subseteq J(F)_{\fq}$. This proves that $X$ is locally Eulerian.
\end{proof}

\begin{Remark}\rm 
The converse of Theorem~\ref{ADE} is not true in general. Consider the reduced curve $X=V(y^4z-x^5+x^2y^3)$. The curve $X$ has one singular point $[0:0:1]$ of multiplicity $4$. The gradient ideal  is codimension $2$ perfect ideal with the syzygy matrix 
\[
\begin{bmatrix}
9y^2& 0\\
15x^2-20yz& -y^2\\
80z^2-18xy&3x^2+4yz
\end{bmatrix}
.\]
Hence the ideal generated by entries of the syzygy matrix is $(x,y,z)$-primary ideal. Therefore, $X$ is a curve of gradient linear type while the singularity is not simple.    
\end{Remark}
The following  example illustrate that all singular points should be simple.   
\begin{Example}
Let $X$ be the  sextic curve defined by the polynomial 
\[f=(x^2-y^2)^3-x^2y^2z^2.\]
The curve $X$ has three singular points. Indeed, $$\mathrm{Sing}(X)=\{[1:1:0],[1:-1:0],[0:0:1]\}.$$ 
The first two singular points are cusp and the last one is quadruple point with two repeated tangents. A calculation shows that the gradient ideal $J(f)$ is locally a complete intersection  at prime ideals corresponds to cusp singular points. Take the singular point $p=[0:0:1]$. The affine curve $X=V(F)\subset \mathbb{A}_k^2$ has one singular points at origin where $F(x,y)=(x^2-y^2)^3-x^2y^2\subset k[x,y]$. A simple computation yields that the Milnor number of $f$ at $p$ is $\mu_{\fm_p}(f)=13$ while the Tjurina number is $\tau_{\fm_p}(f)=12$. Hence $F$ is not locally Eulerian which implies that $X$ is not of gradient linear type.  
\end{Example}

Let $X\subseteq \pp^2$ be a reduced singular curve defining by a reduced polynomial $f\in R=k[x,y,z]$ of degree $d\geq 3$. The topology of $X$ is determined by singularities and  {\it geometric genus}. Let $r_p(X)$ be the number of analytic branches of $X$ at singular point $p$. The $\delta$-invariant is defined by the formula 
\begin{equation*}
\delta_p(X):=\dfrac{1}{2}(\mu_p(X)+r_p(X)-1). 
\end{equation*}
The geometric genus formula asserts that, if  $X\subset\pp_k^2$ be an irreducible plane curve of degree $d$, then the genus of $g(X)$ is given by 
\begin{equation}\label{Genus formula}
g(X)=\dfrac{1}{2}(d-1)(d-2)\ -\sum_{p\in\mathrm{Sing}(X)} \delta_p(X).
\end{equation}
The invariants of some singularities of multiplicity $2$ and $3$ described in table \ref{23}. 
\begin{table}
	\caption{The numerical invariant of singularities} 
	\centering 
	\begin{tabular}{l c c c} 
		Type of singularity & $\mu_p$ & $r_p$ & $\delta_p$ \\ [.5ex] 
		\hline 
		$\mathrm{A}_1$ (Node) & 1& 2 &1 \\ 
		$\mathrm{A}_2$	(simple Cusp) & 2& 1 & 1 \\
		$\mathrm{A}_3$	(Tacnode) & 3 & 2 & 2 \\
		$\mathrm{A}_4$	(Ramphoid cusp)  & 4 & 1 &2 \\
		$\mathrm{A}_5$	(Oscnode) & 5 & 2 & 3 \\ 
		$\mathrm{A}_6$ (higher cusp)  & 6 & 1 & 3 \\
		$\mathrm{D}_4$ (Ordinary triple point) & 4 & 3 & 3 \\
		$\mathrm{D}_5$		(Tacnode cusp) & 5 &2 & 3 \\
		$\mathrm{E}_6$ (Multiplicity $3$ cusp)  & 6 & 1 & 3 \\ [1ex] 
	\end{tabular}
	\label{23} 
\end{table}
 
 Let $X=V(f)\subseteq \mathbb{P}_k^2$ be a reduced cubic curve. Then $X$ is of gradient linear type~\cite[Corollary 3.7]{Abbas}. The second author  and A. Simis proved that any rational irreducible quartic curve is of gradient linear type. The proof is based to find a normal form for each class of singularity according to the nature of singularities~\cite[Theorem 3.12]{AA}. 
 \begin{Theorem}
 	Any reduced singular quartic curve is of gradient linear type. 
 \end{Theorem}
 \begin{proof}
Let $X\subseteq \pp^2$ be an irreducible quartic curve. 	The genus formula (\ref{Genus formula}) shows that 
\[g(X)=3-\sum_{p\in\mathrm{Sing}(X)} \delta_p(X).\]
Since $g(X)\geq 0$ and $\delta_p(X)\geq 1$ for $p\in\mathrm{Sing}(X)$, it follows that  the quartic $X$ has at most three singular points and $0\leq g(X)\leq 3$. If $g(X)=0$, then we have three cases on the number of singularities. If the number of singularities is 3, then $X$ admit  singularities of type three node, three cusp, one node and two cusp, two node and one cusp. If the number of singularities is $2$, then $X$ admit one node and Tacnode, one cup and Tacnode, one node and  Ramphoid cusp, one cusp and  Ramphoid cusp.  If the number of singularities is $1$, then $X$ has the singularities: Oscnode, $\mathrm{A}_6$ cups and one singular points of multiplicity $3$ ( Ordinary triple point, Tacnode cups and multiplicity $3$ cusp). 

Let $g(X)=1$. Note that  $p\in \mathrm{Sing}(X)$ can not be the singularity of type $\mathrm{A}_k$ for $k\geq 5$, because in such case the intersection multiplicity  $\mathrm{mult}_p(X,L)\geq 5$, and this contradicts the assumption that $\deg(X)=4$. We have two cases on the number of singularities. If the number of singularities is $2$, then $X$ has singularities of type: two nodes, two cusps, one node and one cusp. If $X$ has one singular points, then one has one Tacnode and one Ramphoid cusp. 

If $g(X)=1$, then $X$ has just one singular point $p$ with $\delta_p=1$. Thus, it must be a cusp or a node.

Therefore, any irreducible quartic curve just has simple singularities and Theorem(~\ref{ADE}) proves that $X$ is of gradient linear type. 

Now let $X$ be a reduced and reducible quartic. We have also the genus formula for reducible plane curves. It is not easy to make use of the formula as we have done in the irreducible case. The easiest way to find possible  singularities is from the geometric point of view. Geometrically, a quartic can only reduce into the following four kind of combinations:

{\rm (I).} \textbf{A cubic and a line}. If the cubic is irreducible, then the quartic will be included in the other 3 cases. Therefore, we assume that cubic is irreducible. It is well-known that there are only three types of irreducible cubics, namely, non-singular, nodal and cuspidal. 

If the cubic is non-singular, we get the singularities:
	\begin{itemize}
		\item  Three nodes. 
		\item  One node and one Tacnode. 
        \item  Oscnode. 
		\end{itemize}
If the cubic is Nodal, we have the following singularities:
\begin{itemize}
	\item Four nodes. 
	\item Two nodes and one Tacnode. 
	\item One node and one Oscnode.
	\item One node and one ordinary triple point($\mathrm{D}_4$).
	\item $\mathrm{D}_6$ singularity. 
\end{itemize}
Finally, if the cubic is Cuspidal, one has:
\begin{itemize}
	\item Three nodes and one cusp. 
	\item One node and one Tacnode. 
	\item One node, one Oscnode and one Tacnode.
	\item One cusp and one Oscnode.
	\item One node and one Tacnode cusp ($\mathrm{D}_5$). 
	\item $\mathrm{E}_7$ singularity. 
\end{itemize}
{\rm (II).}\textbf{ Two conics}. Two conics can have at most four intersections. Therefore, according to the number of intersections , we get the singularities:
\begin{itemize}
	\item Four nodes. 
	\item Two nodes and one Tacnode. 
	\item One node and one Oscnode.
	\item Two Tacnodes. 
	\item $\mathrm{A}_7$ singularity. 
\end{itemize}
{\rm (III).} \textbf{A conic and two lines}. Two distinct lines can cut each others at exactly one points and each of them can have at most two intersections with the conic. Thus according the number of intersections of these line with the conic, we get the singularities
\begin{itemize}
	\item Five nodes.
	\item Three nodes and one Tacnode. 
	\item One node and two Tacnodes.
	\item Two nodes and one ordinary triple point($\mathrm{D}_4$). 
	\item  One node and one $\mathrm{D}_6$ singularity. 
\end{itemize}
{\rm (IV).}\textbf{ Four lines}. Every $n$ distinct lines can have at most $\frac{n(n+1)}{2}$ intersections. If the lines are distinct, then four lines can have at most six intersections. Then there are three possibilities:
\begin{itemize}
	\item Four distinct lines (Six nodes).
	\item Three concurrent lines and one other (three nodes  and one  ordinary triple point($\mathrm{D}_4$). 
	\item  Four concurrent lines. 
\end{itemize}
Then any reduced reducible quartic only have simple singularities except the case four concurrent lines. By a projective transformation we may assume that $X=V(xz(x+z)(x-z))$ which is clearly locally Eulerian  at the  singular point $[0:1:0]$. Therefore, Theorem~(\ref{ADE}), complete the proof. 
\end{proof}
 \section*{Acknowledgment} The authors would like to thank Rashid Zaare-Nahandi for reading the manuscript and helpful comments on the subject to make the paper more comprehensive. 
 
\end{document}